\documentclass[letterpaper, 10 pt, conference,twocolumn]{ieeeconf}
\IEEEoverridecommandlockouts

\usepackage{amsmath,amsfonts,amssymb}
\usepackage{graphicx,epsfig,psfrag,subfigure}
\usepackage{xfrac,bm}
\usepackage{algorithmic,algorithm}
\usepackage[noadjust]{cite}

\IEEEoverridecommandlockouts 

\DeclareMathOperator*{\argmin}{arg\,min}

\DeclareMathOperator{\proj}{\mathbf{Proj}}

\newcommand{\x}{{\bm{x}}}

\newcommand{\z}{{\bm{z}}}

\newcommand{\p}{{\bm{p}}}

\newcommand{\uu}{{\bm{u}}}

\newcommand{\mbx}{\mathbf{x}}
\newcommand{\mbz}{\mathbf{z}}
\newcommand{\mbf}{\mathbf{f}}

\newcommand{\mbT}{\mathbf{T}}
\newcommand{\cA}{\mathcal{A}}
\newcommand{\cB}{\mathcal{B}}

\newcommand{\cH}{\mathcal{H}}

\newcommand{\cO}{\mathcal{O}}

\newcommand{\cR}{\mathcal{R}}

\newcommand{\cT}{\mathcal{T}}
\newcommand{\cU}{\mathcal{U}}

\newcommand{\bR}{\mathbb{R}}

\newcommand{\dist}{{\mathrm{dist}}}

\DeclareSymbolFont{symbolsC}{U}{pxsyc}{m}{n}

\newtheorem{remark}{Remark}
\newtheorem{proposition}{Proposition}
\newtheorem{corollary}{Corollary}

\newtheorem{theorem}{Theorem}

\newtheorem{example}{Example}
\newtheorem{lemma}{Lemma}

\usepackage{xcolor}

\renewcommand{\baselinestretch}{0.95}

\def\BibTeX{{\rm B\kern-.05em{\sc i\kern-.025em b}\kern-.08em
    T\kern-.1667em\lower.7ex\hbox{E}\kern-.125emX}}

\title{\textbf{Guarding a Target Set from a Single Attacker in the Euclidean Space}}

\author{Yoonjae Lee \and Efstathios Bakolas \thanks{Y. Lee (graduate student) and E. Bakolas (Associate Professor) are with the Department of Aerospace Engineering
and Engineering Mechanics, The University of Texas at Austin,
Austin, Texas 78712-1221, USA. Emails: yol033@utexas.edu; bakolas@austin.utexas.edu}}

\begin{document}

\maketitle

\begin{abstract}
This paper addresses a two-player target defense game in the $n$-dimensional Euclidean space where an attacker attempts to enter a closed convex target set while a defender strives to capture the attacker beforehand. We provide a complete and universal differential game-based solution which not only encompasses recent work associated with similar problems whose target sets have simple, low-dimensional geometric shapes,
but can also address problems that involve nontrivial geometric shapes of high-dimensional target sets. The value functions of the game are derived in a semi-analytical form that includes a convex optimization problem. When the latter problem has a closed-form solution, one of the value functions is used to analytically construct the barrier surface that divides the state space of the game into the winning sets of players. For the case where the barrier surface has no analytical expression but the target set has a smooth boundary, the bijective map between the target boundary and the projection of the barrier surface is obtained. By using Hamilton-Jacobi-Isaacs equation, we verify that the proposed optimal state feedback strategies always constitute the game's unique saddle point whether or not the optimization problem has a closed-form solution. We illustrate our solutions via numerical simulations.
\end{abstract}

\section{Introduction}

The problem of guarding a target has drawn a lot of attention due to its relevance to a wide spectrum of applications in aerospace, military, and robotics. Differential game theory, founded by Isaacs in \cite{isaacs1965differential}, provides a rigorous framework to analyze this type of problems by treating adversarial interactions between agents as a dynamic game subject to their kinodynamic constraints \cite{bacsar1998dynamic}. In his pioneering work \cite{isaacs1965differential}, Isaacs addresses a two-player differential game of guarding a planar target area and suggests a geometric method with which one can determine the possibility of capture of the attacker and find optimal guarding/attacking strategies when both players have simple motion and the same speed.

Target defense (or reach-avoid) differential games have been explored extensively, often with a point target (singleton) \cite{li2011defending,garcia2020optimal,selvakumar2019feedback}.
In \cite{li2011defending}, Li \textit{et al.} study a two-player planar game of guarding a point target or asset using the framework of linear quadratic differential games. The saddle-point solution of a two-defender single-attacker reach-avoid game with a point target in three-dimensional space has been studied by Garcia \textit{et al.} \cite{garcia2020optimal}. In \cite{selvakumar2019feedback}, Selvakumar \textit{et al.} develop a feedback strategy for a single attacker to reach a point target against a team of distributed defenders employing the relay pursuit strategy~\cite{bakolas2012relay}.

In practice, a target is often better described by a set which is not a singleton, such as a border line \cite{yan2018reach,yan2019task,garcia2020multiple,yan2020guarding} or a closed (or compact) set
\cite{yan2017defense,garcia2019optimal,pachter2017differential,von2020multiple,shishika2020cooperative}. In \cite{yan2018reach,yan2019task,garcia2020multiple}, the authors provide analytical solutions to multiplayer border defense games. Guarding a subspace in the $n$-dimensional space has been recently studied by Yan \textit{et al.} \cite{yan2020guarding}. In \cite{yan2017defense}, the same authors address a two-player perimeter defense game, which was later revisited by Garcia \textit{et al.} \cite{garcia2019optimal}. In \cite{pachter2017differential} and \cite{von2020multiple}, the authors apply Isaacs' method to a polygonal target. 
In \cite{shishika2020cooperative}, Shishika \textit{et al.} discuss a class of perimeter defense games in which the motion of defenders is constrained along the perimeter of a planar target area.
Experimental implementation of target defense games is presented in~\cite{mohanan2018toward} and \cite{fu2020}.

Although all of the aforementioned previous work offer novel solutions to certain classes of target defense games, these solutions in general are limited to problems in which the target set has a specific type of geometry. The motivation of this paper is to formulate a global, standalone solution model that is applicable to a broader pool of target defense games. The solution proposed in this paper, mainly focused on the Two-Player Target Defense Game (TPTDG) in the $n$-th dimensional space, is designed to not only solely produce the same results (i.e., barrier and saddle point strategies) as some of the previous work but also provide answers to more challenging, unexplored problems in which an arbitrary shape of high-dimensional convex target is to be guarded.

Our approach follows the rigorous differential game-based solution procedure adopted in, for example, \cite{garcia2020optimal} in that the Game of Kind and the Game of Degree \cite{isaacs1965differential} of the TPTDG are sequentially addressed. Taking over unresolved problems from our previous work \cite{lee2021optimal}, this paper provides an explicit statement about when the proposed semi-analytical barrier function and optimal guarding/attacking strategies admit closed-form expressions, which draws a direct connection with convex optimization. We also demonstrate how to apply our solution to a few special examples of TPTDG in which the differential game-based solutions are in closed-form as in, for instance, \cite{garcia2020optimal,yan2020guarding}, and \cite{pachter2017differential}, thereby verifying the universality of our solution. Furthermore, the bijective map between the target boundary and the projection of the barrier surface is derived to address non-trivial cases for which no analytical solution exists. The proposed optimal strategies are verified by Isaacs' method, namely the Hamilton-Jacobi-Isaacs (HJI) equation, to always correspond to the saddle point solution of the game, provided a convex target set.

The rest of the paper is structured as follows. In Section~\ref{sec:problem_formulation}, the $n$-dimensional TPTDG is formulated. In Section~\ref{sec:gameofkind}, the barrier function is defined and then the barrier surface and the winning sets of players are identified. In Section~\ref{sec:gamesofdegree}, optimal guarding/attacking strategies are characterized. In Section~\ref{sec:simulations}, illustrative numerical simulations are presented. Finally, concluding remarks are presented in Section~\ref{sec:conclusion}.

\section{Problem Formulation} \label{sec:problem_formulation}

\subsection{Notation}
The symbols $\mathbb{N}$, $\bR$, $\bR_{\geq 0}$, and $\bR_{>0}$ denote the set of natural numbers (positive integers), real numbers, nonnegative real numbers, and positive real numbers, respectively. The operator $\langle \cdot,\cdot \rangle$ (resp., $\|\cdot\|$) denotes the inner product (resp., $\ell^2$-norm) in the Euclidean space $\bR^n$ (where $n \in \mathbb{N}$). Given a point $\x \in \bR^n$ and a closed set $\cO \subset \bR^n$, we define the (set) distance function $\dist : \bR^n \times \cO \rightarrow \bR_{\geq 0}$ that measures the closeness of $\x$ from $\cO$ by $\mathrm{dist}(\x, \cO) := \min_{\z \in \cO}\|\x - \z\|$. If $\cO$ is closed and convex, the projection operator $\proj_\cO : \bR^n \rightarrow \cO$ which we define by $\proj_{\cO}(\x) := \argmin\nolimits_{\z \in \cO} \| \x - \z \|$ computes the (unique) orthogonal projection of $\x$ onto $\cO$.

\subsection{Target-Defense Differential Game in $\bR^n$}
The TPTDG between a \textit{defender} or \textit{pursuer} ($P$) and an \textit{attacker} or \textit{evader} ($E$) is considered. The game space is defined by the $n$-dimensional Euclidean space $\bR^n$ where $n \in \mathbb{N}$. Let $\Omega$ denote a nonempty, closed and convex subset of $\bR^n$, referred to as the Target Set (TS), which $E$ attempts to attack while $P$ strives to guard. TS is assumed to be permeable for $P$, i.e., $P$ can move in and out of TS freely. At every time instant, both players have complete information about each other's state. The kinematics of the agents are given by
\begin{equation} \label{eq:motion}
    \begin{aligned}
       \dot{\x}_{P}  &= v_{P} \uu_{P}, \qquad &\x_{P}(0) & =\x_{P}^0,
        \\
        \dot{\x}_E  &= v_E \uu_E,  &\x_E(0) & =\x_{E}^0,
    \end{aligned}
\end{equation}
where $\x_P \in \bR^n$ and $\x_E \in \bR^n$ (resp., $\x_{P}^0 \in \bR^n$ and $\x_{E}^0 \in \bR^n$) denote the position of $P$ and $E$ at time $t$ (resp., at time $t=0$), respectively. Similarly, $\uu_P \in \cU$ and $\uu_E \in \cU$ (resp., $v_P \in \bR_{>0}$ and $v_E \in \bR_{>0}$) denote the control input (resp., maximum allowable speed) of $P$ and $E$, respectively, where $\cU := \{ \uu \in \bR^n : \| \uu \| = 1 \}$. Let $\gamma := v_E/v_P$ (speed ratio), then throughout the paper it is assumed that $0 < \gamma < 1$ (i.e., $E$ is slower than $P$). Denote the game state by $\mbx = (\x_P, \x_E) \in \mathbb{R}^{2n}$, then the game dynamics can be written as
\begin{equation}\label{eq:statedyns}
\dot{\mbx} = \mbf(\mbx,\uu_P,\uu_E),~~~~\mbx(0)=\mbx_0,
\end{equation}
where $\mbx_0 = (\x_{P}^0, \x_{E}^0) \in \bR^{2n}$ is the initial state of the game and $\mbf: \bR^{2n} \times \cU \times \cU \rightarrow \bR^{2n}$ is the vector field of the game dynamics, where $\mbf(\mbx,\uu_P,\uu_E) := (v_P \uu_P, v_E \uu_E)$.

The TPTDG terminates if either capture or attack occurs with zero proximity. The final time of the game is $t_f := \inf \{ t \in \bR_{>0}: (\mbx(t) \in \cT_c) \vee (\mbx(t) \in \cT_a) \}$, where the capture and attack terminal manifolds are defined by
\begin{align}
    \cT_c &:= \left\{ \mbz \in \bR^{2n}: \|\x_P - \x_E\| = 0 \right\}, \label{eq:captermanifold}
    \\
    \cT_a &:= \left\{ \mbz \in \bR^{2n} : \x_E \in \Omega \right\}. \label{eq:atttermanifold}
\end{align}
At $t_f$, $P$ (resp., $E$) wins the game if $\mbx(t_f) \in \cT_c$ (resp., if $\mbx(t_f) \in \cT_a$). As will be discussed in the following section, the state space of the game, $\bR^{2n}$, can be divided by the barrier surface $\cB$ into two disjoint sets $\cR_c$ and $\cR_a$ in each of which a different local game (or subgame) is played. In particular, if $\mbx_0 \in \cR_c$, capture of $E$ is ensured under optimal play, so the objective of $P$ (resp., $E$) is to maximize (resp., minimize) the minimum distance between the point of capture and TS, which induces a subgame called the capture game whose payoff functional is defined by
\begin{equation}\label{eq:payofffun1}
    J_c \left( \uu_P(\cdot),\uu_E(\cdot) \right) := \dist ( \x_E(t_f),\Omega ),
\end{equation}
where $\uu_P(\cdot)$ and $\uu_E(\cdot)$ denote the  state feedback strategies of $P$ and $E$. The value function of this game is
\begin{equation} \label{eq:valfun1}
    V_{c}(\mbx_0) := \min\nolimits_{\uu_E(\cdot)} \max\nolimits_{\uu_P(\cdot)} J_c,
\end{equation}
subject to \eqref{eq:statedyns} and \eqref{eq:captermanifold}. If $\mbx_0 \in \cR_a$, on the other hand, capture is not possible, so the objective of $P$ (resp., $E$) is to minimize (resp., maximize) the distance between $P$'s final position and the point of attack. This subgame is referred to as the attack game and its payoff functional is defined by
\begin{equation} \label{eq:payofffun2}
    J_a( \uu_P(\cdot),\uu_E(\cdot) ) := \left\| \x_E(t_f) - \x_P(t_f) \right\|.
\end{equation}
The value function of the attack game is
\begin{equation}\label{eq:valfun2}
    V_a(\mbx_0) := \min\nolimits_{\uu_P(\cdot)} \max\nolimits_{\uu_E(\cdot)} J_a,
\end{equation}
subject to \eqref{eq:statedyns} and \eqref{eq:atttermanifold}.

\section{Game of Kind} \label{sec:gameofkind}

Addressed in this section is the Game of Kind in TPTDG. Specifically, the barrier surface of the game, which demarcates the availability of capture and divides the state space of the game into the winning sets of $P$ and $E$, is identified.

\subsection{Isaacs' Geometric Method and Barrier Function}

\begin{proposition}\label{prop:hamiltonian}
    The saddle-point strategies (in open-loop form) of the TPTDG defined in Section~\ref{sec:problem_formulation} correspond to constant inputs over time and their corresponding optimal trajectories are straight lines.
\end{proposition}
\begin{proof}
    The Hamiltonian of either subgame is given by $H = v_P \bm{\lambda}_P^\top \uu_P + v_E \bm{\lambda}_E^\top \uu_E$,
    where $\bm{\lambda}_P \in \bR^n$ and $\bm{\lambda}_E \in \bR^n$ are the co-state vectors. Since payoff functionals~\eqref{eq:payofffun1} and \eqref{eq:payofffun2} are both Mayer-type, $\dot{\bm{\lambda}}_P = \bm{0}$ and $\dot{\bm{\lambda}}_E = \bm{0}$, which implies that $\bm{\lambda}_P$ and $\bm{\lambda}_E$ are constant. Since $H$ is separable in $\uu_P$ and $\uu_E$, Isaacs' condition holds, and thus Pontryagin's (min-max) principle can be applied to derive optimal strategies in open-loop form \cite{bacsar1998dynamic}, which correspond to constant inputs in this case. Consequently, the corresponding optimal trajectories are straight lines and the proof is complete.
\end{proof}

Since the optimal trajectories of both players are straight lines, we can utilize Isaacs' geometric method \cite{isaacs1965differential}. 

\begin{lemma} \label{lemma:apolsphere}
    Given the TPTDG defined in Section~\ref{sec:problem_formulation}, let the Safe Region (SR) and the Boundary of the Safe Region (BSR) of $E$, which respectively refer to the interior and boundary of the set of all points in $\bR^n$ that $E$ can reach without being captured by $P$, be defined by
    \begin{align}
        \cA(\mbx;\gamma) &:= \left\{ \z \in \bR^n : \|\z-\bm{\alpha}(\mbx;\gamma)\| < \beta(\mbx;\gamma) \right\},
        \\
        \partial \cA(\mbx;\gamma) &:= \left\{ \z \in \bR^n : \|\z-\bm{\alpha}(\mbx;\gamma)\| = \beta(\mbx;\gamma) \right\},
    \end{align}
    where
    \begin{align}
        \bm{\alpha}(\mbx;\gamma) = \frac{\x_{E} - \gamma^2 \x_{P}}{1-\gamma^2},~\beta(\mbx;\gamma) = \frac{\gamma \left\| \x_{E} - \x_{P} \right\|}{1-\gamma^2}. \label{eq:apolcenterandradius}
    \end{align}
    Then, $P$ can win the game if $\Omega \cap \cA(\mbx_0;\gamma) = \varnothing$ whereas $E$ can win the game if $\Omega \cap \cA(\mbx_0;\gamma) \neq \varnothing$ (under optimal play).
\end{lemma}
\begin{proof}
    According to Isaacs' geometric method, the SR (resp., BSR) corresponds to the interior (resp., boundary) of the Apollonius circle associated with the positions of $P$ and $E$, which at $t=0$ is defined as the set of points that satisfy
    \begin{align}
        \left\|\z-\x_E^0\right\| = \gamma \left\|\z-\x_P^0\right\|, ~\z \in \bR^n.
    \end{align}
    Squaring both sides and rearranging terms gives
    \begin{align}\label{eq:proofeq1}
        \|\z\|^2 - \frac{2 \left\langle \z, \x_E^0 - \gamma^2 \x_P^0 \right\rangle}{1-\gamma^2} = \frac{\gamma^2\left( \|\x_P^0\|^2 - \|\x_E^0\|^2 \right)}{1-\gamma^2}.
    \end{align}
   By adding $\|\x_E^0-\gamma^2 \x_P^0\|^2/(1-\gamma^2)^2$ to both sides of \eqref{eq:proofeq1} and taking the square root, it follows readily that
    \begin{align}
        \left\| \z - \frac{\x_E^0 - \gamma^2 \x_P^0}{1-\gamma^2} \right\| = \frac{\gamma \left\|\x_E^0-\x_P^0\right\|}{1-\gamma^2},
    \end{align}
    which proves that $\bm{\alpha}(\mbx;\gamma)$ and $\beta(\mbx;\gamma)$ satsify Eq.~\eqref{eq:apolcenterandradius}. It then follows from the definitions of SR that $E$ can (resp., cannot) enter TS before being captured by $P$ if TS and SR intersect (resp., are disjoint).
\end{proof}

For notational brevity, the ratio $\gamma$ will be dropped from the arguments of $\bm{\alpha}$ and $\beta$ throughout the paper. Using the notions of SR and BSR, the barrier function and the winning sets of players \cite{garcia2020optimal} are characterized in the following theorem.

\begin{theorem} \label{theo:barriersurface}
    Given the TPTDG defined in Section~\ref{sec:problem_formulation}, the (semipermeable) barrier surface, $\cB$, and the winning sets of players separated by $\cB$, namely $\cR_c$ (where capture is ensured) and $\cR_a$ (where attack is ensured), are given by
    \begin{align}
        \cB &:= \left\{ \mbz \in \bR^{2n} : B(\mbz;\Omega,\gamma) = 0 \right\},
        \\
        \cR_c &:= \left\{ \mbz \in \bR^{2n} : B(\mbz;\Omega,\gamma) > 0 \right\}, \label{eq:Pwinningset}
        \\
        \cR_a &:= \left\{ \mbz \in \bR^{2n} : B(\mbz;\Omega,\gamma) < 0 \right\}, \label{eq:Ewinningset}
    \end{align}
    where the barrier function $B : \bR^{2n} \rightarrow \bR$ is defined by
    \begin{align} \label{eq:barrierfunction}
        &B(\mbx;\Omega,\gamma) := \left\|  \bm{\alpha}(\mbx)-\proj_\Omega (\bm{\alpha}(\mbx)) \right\| - \beta(\mbx).
    \end{align}
    Then, under optimal play, $P$ can win the game if $\mbx_0 \in \cR_c$, whereas $E$ can win the game if $\mbx_0 \in \cR_a$.
\end{theorem}
\begin{proof}
    The barrier function is an indicator-like function that attains a positive value if capture is possible and a negative value otherwise. Let us choose the value function~\eqref{eq:valfun1} as our barrier function (the reason will be clarified along the proof), then in view of Lemma~\ref{lemma:apolsphere} the barrier function, or equivalently Eq.~\eqref{eq:valfun1}, is simplified as
    \begin{align} \label{eq:Vc}
        B(\mbx_0;\Omega,\gamma) &= \min\nolimits_{\uu_E^\star(\cdot)}\max\nolimits_{\uu_P^\star(\cdot)} \dist(\x_E(t_f),\Omega) \nonumber
        \\
        &= \min\nolimits_{\z \in \Omega} \| \bm{\alpha}(\mbx_0)-\z \| - \beta(\mbx_0) \nonumber
        \\
        &= \| \bm{\alpha}(\mbx_0)-\proj_\Omega (\bm{\alpha}(\mbx_0)) \| - \beta(\mbx_0),
    \end{align}
    where the existence and uniqueness of the projection are ensured by the fact that TS is a closed and convex set. It then follows from Lemma~\ref{lemma:apolsphere} that $B(\mbx_0;\Omega,\gamma) < 0 \Leftrightarrow \proj_\Omega (\bm{\alpha}(\mbx_0)) \in \cA(\mbx_0;\gamma)$ and $B(\mbx_0;\Omega,\gamma) > 0 \Leftrightarrow \proj_\Omega (\bm{\alpha}(\mbx_0)) \notin \cA(\mbx_0;\gamma)$. Since by definition $\proj_\Omega ( \cdot ) \in \Omega$, the two aforementioned statements imply that $B(\mbx_0;\Omega,\gamma) < 0 \Leftrightarrow \Omega \cap \cA(\mbx_0;\gamma) \neq \varnothing$ and $B(\mbx_0;\Omega,\gamma) > 0 \Leftrightarrow \Omega \cap \cA(\mbx_0;\gamma) = \varnothing$. Then, in view of Lemma~\ref{lemma:apolsphere} again, under optimal play, the former statement implies that $P$ can win the game, whereas the latter statement implies that $E$ can win the game.
\end{proof}

\begin{corollary}
    Let $P$'s initial position, $\x_P^0$, be fixed, then the Projection of the Barrier Surface (PBS) onto the initial configuration of game space, $\breve \cB(\mbx_0;\Omega,\gamma) \subset \bR^n$, divides the latter space into the (disjoint) subsets $\breve \cR_c$ (projection of $P$'s winning set) and $\breve \cR_a$ (projection of $E$'s winning set). Then, under optimal play, if $\x_E^0 \in \breve \cR_c$, $P$ can win the game, whereas if $\x_E^0 \in \breve \cR_a$, $E$ can win the game.
\end{corollary}


\begin{remark}
    The barrier function~\eqref{eq:barrierfunction} involves a projection operator (i.e., optimization problem), which may or may not admit a closed-form solution, depending on the shape of TS.
\end{remark}


\subsection{Analytical Derivation of Barrier for Simple TS}

Here, we demonstrate a few special examples where one can derive the PBS of TPTDG analytically from Eq.~\eqref{eq:barrierfunction}, i.e., when TS has ``simple enough'' geometry, and compare our results with the related work.

\begin{example}[Singleton] \label{ex:singleton}
    Let TS be the singleton whose unique element is the origin, i.e., $\Omega_\mathrm{ST} := \left\{ \bm{0}_{n} \right\}$, then the projection of a point $\x \in \bR^n \backslash \Omega_\mathrm{ST}$ onto the set $\Omega_\mathrm{ST}$ is the origin itself. Eq.~\eqref{eq:barrierfunction} leads to
    \begin{align}
        \breve \cB(\x_P^0;\Omega_\mathrm{ST},\gamma) = \left\{ \z \in \bR^n \backslash \Omega_\mathrm{ST} : \|\z\| = \gamma \left\| \x_P^0 \right\| \right\},
    \end{align}
    which implies that the winning set of $E$ is the $n$-dimensional open ball centered at the origin with radius $\gamma \|\x_P^0\|$, which is in agreement with the result presented in \cite{selvakumar2019feedback} and could also be applied to \cite{garcia2020optimal} if $E$ therein was slower than $P$.
\end{example}

\begin{example}[Half-Space] \label{ex:halfspace}
Let TS be the closed half-space $\Omega_\mathrm{HS}^- := \{\z \in \bR^n : [\bm{0}_{n-1}^\top,1] \z \leq 0\}$ which is separated by the hyperplane $\cH := \{\z \in \bR^n : [\bm{0}_{n-1}^\top,1] \z = 0\}$. Since the projection of a point $\x \in \bR^n \backslash \Omega_\mathrm{HS}^-$ onto $\Omega_\mathrm{HS}^-$ belongs to $\cH$, we have that $\proj_{\Omega_\mathrm{HS}^-} (\x) = \mathrm{diag}[\bm{1}_{n-1}^\top,0] \x$. After substituting this into Eq.~\eqref{eq:barrierfunction} and some straightforward algebraic manipulation, one can obtain the equation of (the upper sheet of) the hyperboloidal PBS:
\begin{align} \label{eq:planePBS}
    &\breve \cB(\x_{P}^0;\Omega_\mathrm{HS}^-,\gamma) = \Big\{ \z \in \bR^n \backslash \Omega_\mathrm{HS}^- : 1= \nonumber
    \\
    & \left(\z-\mathrm{diag}[\bm{0}_{n-1}^\top,1] \x_P^0\right)^\top \bm{\Lambda} \left(\z-\mathrm{diag}[\bm{0}_{n-1}^\top,1] \x_P^0\right), \nonumber
    \\
    &~ \bm{\Lambda} := \mathrm{diag}\left[-(1-\gamma^2)^{-1} \bm{1}_{n-1}^\top,\gamma^{-2}\right]/\left(x_{P,n}^0\right)^2 \Big\},
\end{align}
where $x_{P,n}^0$ is the $n$-th element of $\x_P^0$. This result agrees with the solution for an one-to-one game presented in \cite{yan2020guarding}.
\end{example}

\begin{example}[Norm Ball] \label{ex:ball}
    Consider the closed norm ball centered at the origin with radius $r$, i.e., $\Omega_\mathrm{CB} := \{ \z \in \bR^n : \|\z\|^2 \leq r \}$. Given a point $\x \in \bR^n \backslash \Omega_\mathrm{CB}$, the three points $\x$, $\proj_{\Omega_\mathrm{CB}}(\x)$, and the origin are all colinear, so $\proj_{\Omega_\mathrm{CB}}(\x) = r \x/\|\x\|$. Then,
    Eq.~\eqref{eq:barrierfunction} can be rewritten as $B(\mbx_0;\Omega_\mathrm{CB},\gamma) = \|\bm{\alpha}(\mbx_0)\| - \beta(\mbx_0) - r$, and PBS (in an implicit form) is accordingly obtained as
    \begin{align} \label{eq:ballPBS}
        \breve \cB(\x_{P}^0;\Omega_\mathrm{CB},\gamma) &= \Big\{ \z \in \bR^n \backslash \Omega_\mathrm{CB} : \left\| \z-\gamma^2\x_{P}^0 \right\| \nonumber
        \\
        &\quad- \gamma \left\|\z-\x_{P}^0\right\| - (1-\gamma^2)r = 0 \Big\}.
    \end{align}
    When $n=2$, the explicit form of Eq.~\eqref{eq:ballPBS} is the Cartesian oval (or, in this particular case, Pascal's limacon) \cite{yan2017defense,garcia2019optimal,pachter2017differential}.
\end{example}

We have therefore shown that some of the previous work (e.g., \cite{garcia2020optimal,yan2020guarding,pachter2017differential}) have a common solution structure which is induced by the generalized barrier function~\eqref{eq:barrierfunction}. Three-dimensional illustrations of Examples~\ref{ex:singleton}, \ref{ex:halfspace}, and \ref{ex:ball} are provided in Figure~\ref{fig:barriers}.

\begin{figure}
    \centering
    \vspace{4mm}
    \includegraphics[scale=0.26]{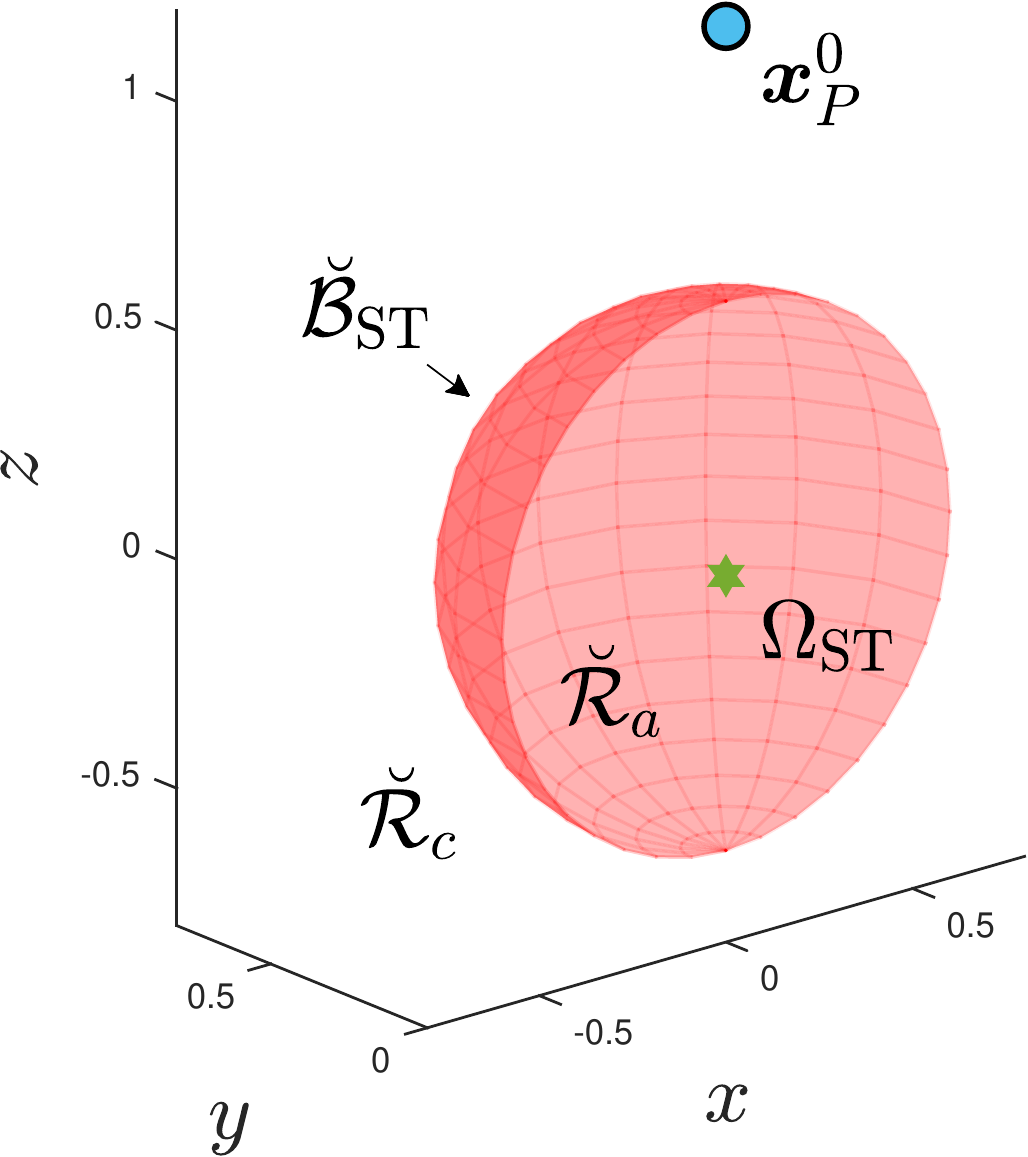}
    \qquad
    \includegraphics[scale=0.26]{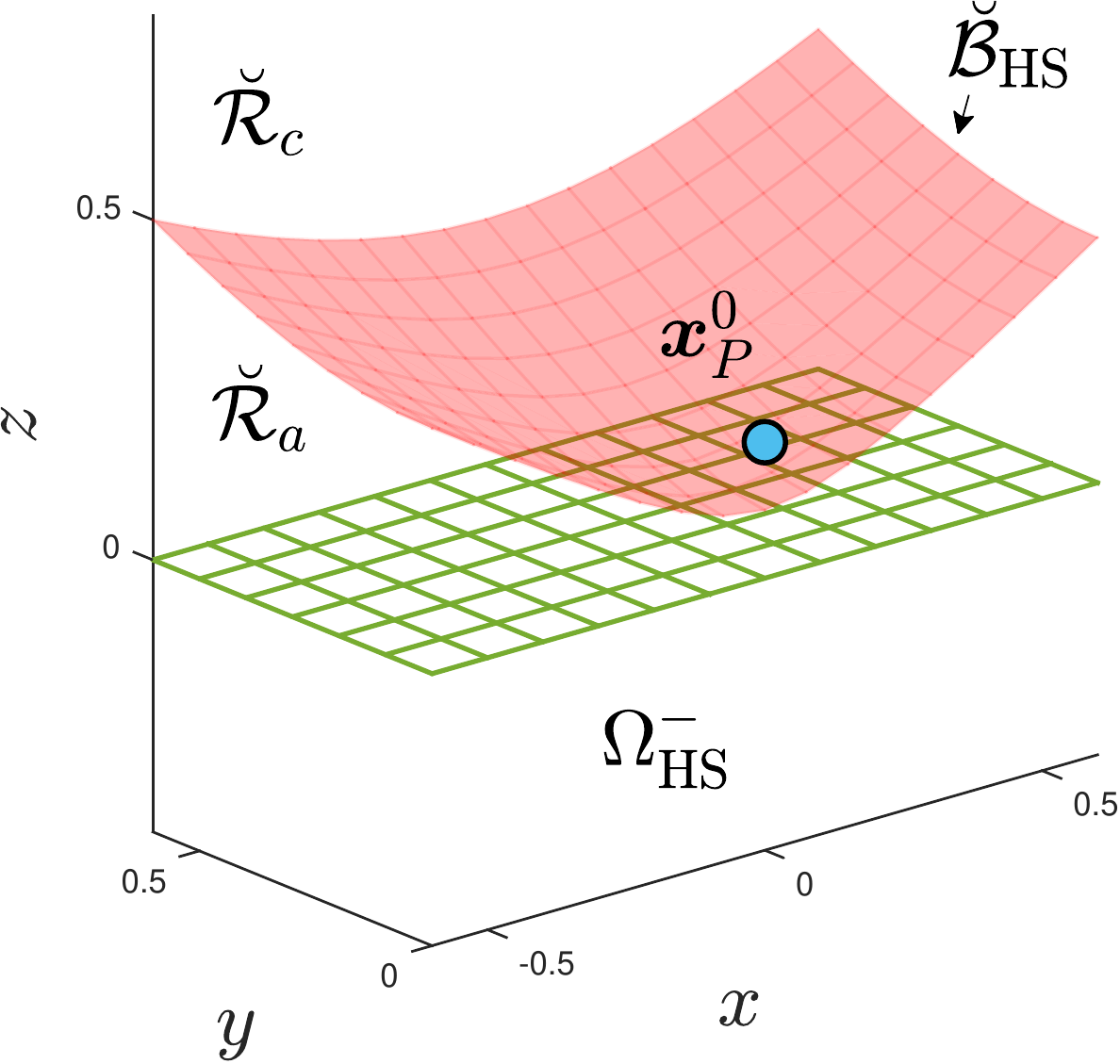}
    \includegraphics[scale=0.24]{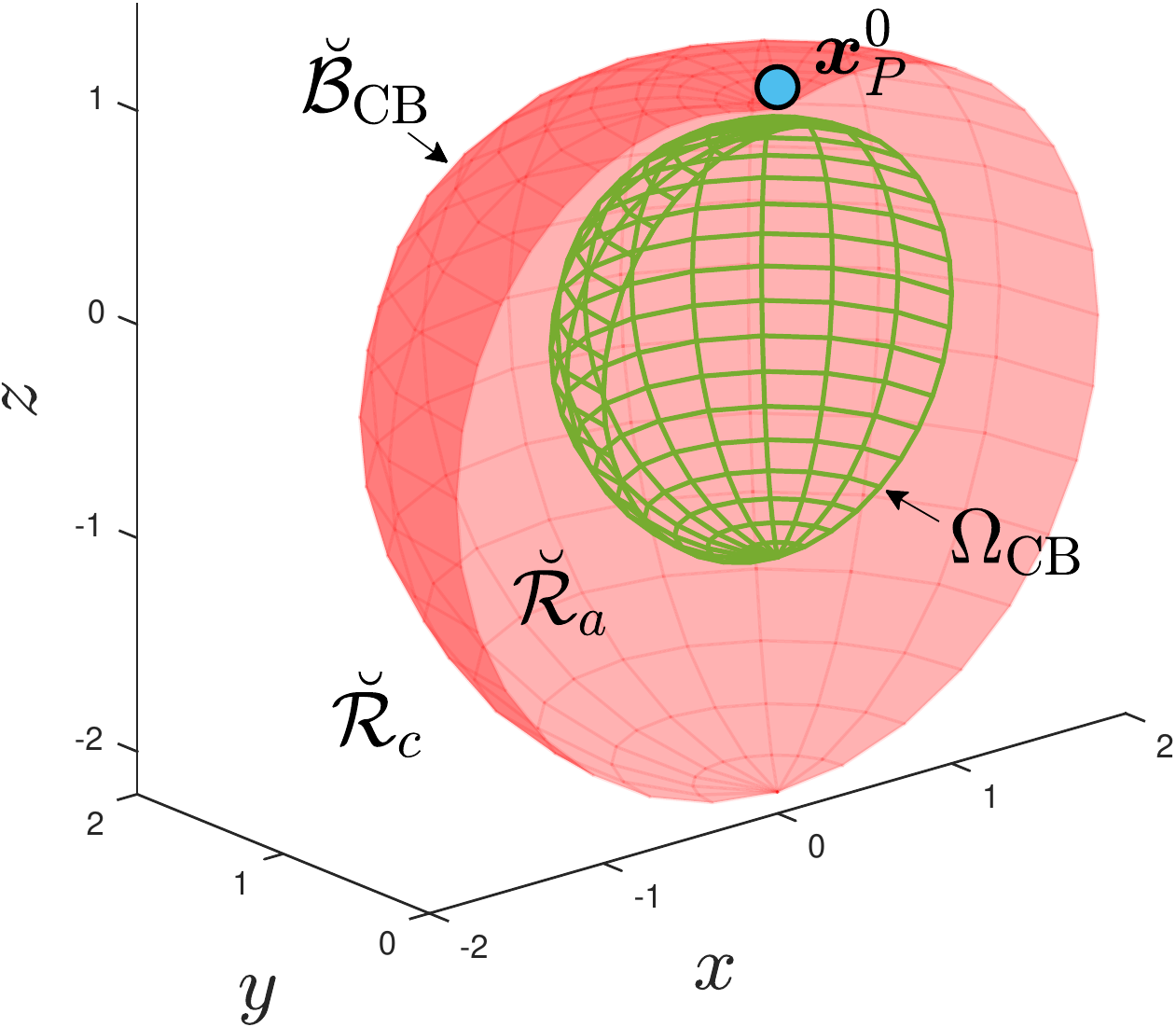}
    \qquad
    \includegraphics[scale=0.24]{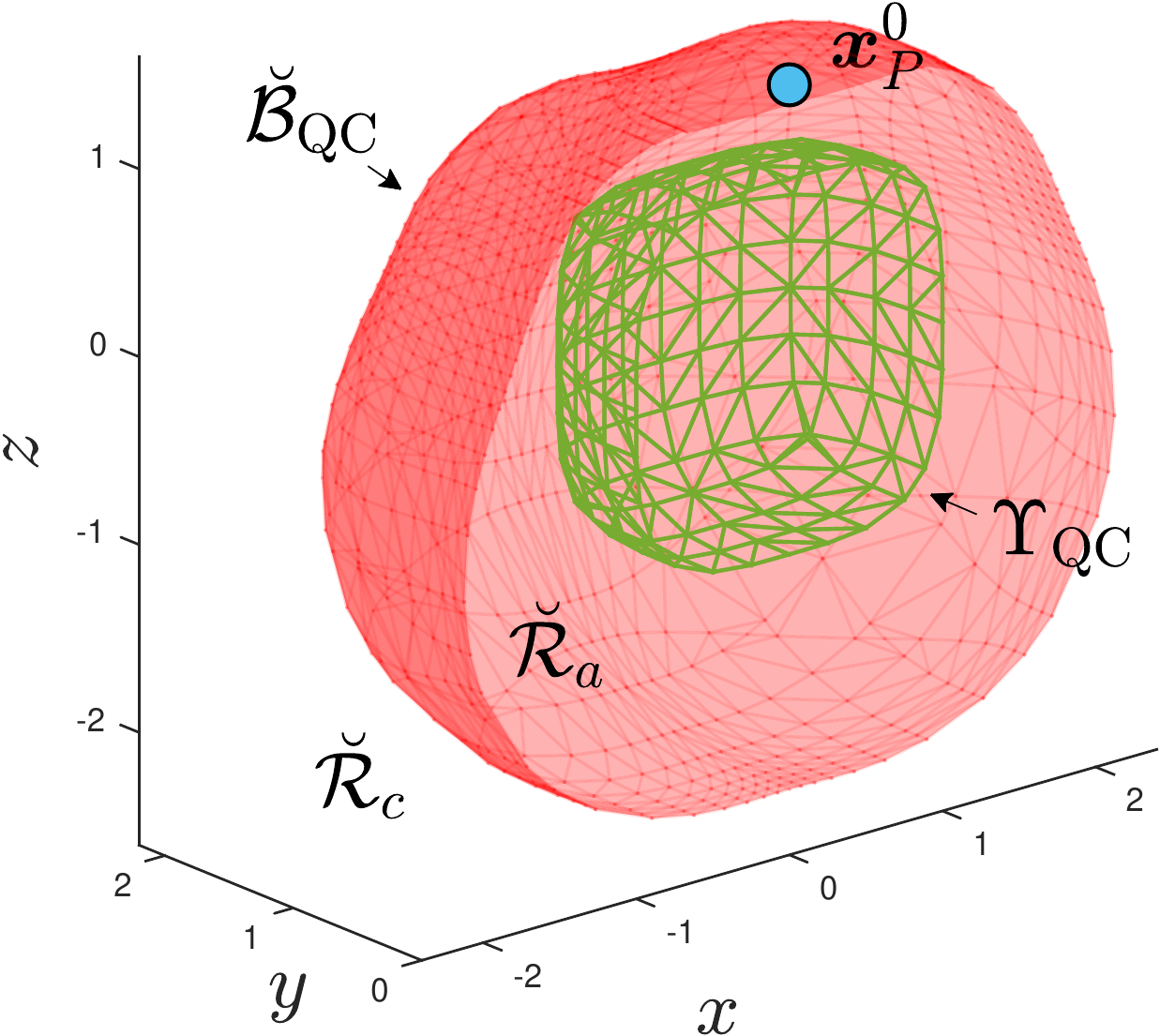}
    \caption{Three-dimensional PBS for TS given by 1) the singleton $\Omega_\mathrm{ST} := \{ \bm{0}_{3} \}$ (top left), 2) the closed lower half-space $\Omega_\mathrm{HS}^- := \{ (x,y,z) \in \bR^3 : z \leq 0 \}$ (top right), 3) the closed ball $\Omega_\mathrm{CB} := \{ (x,y,z) \in \bR^3 : x^2 + y^2 + z^2 \leq 1 \}$ (bottom left) and 4) the quartic cube $\Upsilon_\mathrm{QC} := \{(x,y,z) \in \bR^3 : x^4 + y^4 + z^4 \leq 1 \}$ (bottom right), all with speed ratio $\gamma = 0.5$. Only the half-section of TS and PBS are shown for clear visibility. PBS in the first three examples ($\breve \cB_\mathrm{ST}$, $\breve \cB_\mathrm{HS}$, and $\breve\cB_\mathrm{CB}$) has an analytical expression derived from Eq.~\eqref{eq:barrierfunction}, whereas in the last example it is obtained via Eq.~\eqref{eq:barriermap}.}
    \label{fig:barriers}
\end{figure}


\subsection{Barrier Transformation for TS with Smooth Surface}

If the projection operator has no closed-form expression (i.e., TS has nontrivial shape), PBS cannot be analytically derived from Eq.~\eqref{eq:barrierfunction}. We will show that, however, if TS has a smooth boundary surface, there exists a bijective map between PBS and the boundary surface of TS.
\begin{theorem}
    Given the TPTDG defined in Section~\ref{sec:problem_formulation}, let $\Upsilon$ (resp., $\partial \Upsilon$) denote a closed convex TS (resp., boundary surface of TS), where $\partial \Upsilon$ can be characterized by a convex and smooth function $F : \bR^n \rightarrow \bR$, i.e., $\partial \Upsilon := \{\z \in \bR^n : F(\z) = 0\}$. Let $P$'s initial position, $\x_{P}^0$, be fixed, then the (bijective) function $\mbT_{\partial \Upsilon \mapsto \breve \cB} : \partial \Upsilon \rightarrow \breve\cB$ defined by
    \begin{align} \label{eq:barriermap}
        &\mbT_{\partial \Upsilon \mapsto \breve \cB}(\partial \Upsilon;\x_P^0,\gamma) := \gamma^2 \x_P^0 + (1-\gamma^2) \x + \frac{\nabla F(\x)}{\| \nabla F(\x) \|^2} \cdot \nonumber
        \\
        &~~ \bigg[ -\gamma^2 \left \langle \x_P^0-\x,\nabla F(\x) \right \rangle + \gamma \Big( \gamma^2 \left \langle \x_P^0-\x,\nabla F(\x) \right \rangle^2 \nonumber
        \\
        &~~ - (1-\gamma^2) \| \nabla F(\x) \|^2 \left\|\x_P^0-\x\right\|^2 \Big)^{\frac{1}{2}} \bigg],~~\x \in \partial \Upsilon
    \end{align}
    maps $\partial \Upsilon$ to PBS, i.e., $\breve{\cB}(\x_P^0;\Upsilon,\gamma) = \mbT_{\partial \Upsilon \mapsto \breve \cB}(\partial \Upsilon; \x_P^0,\gamma)$.
\end{theorem}
\begin{proof}
    Let us choose an arbitrary point of the boundary of TS, $\p \in \partial \Upsilon$, and assume that $\x_P(t_f) = \x_E(t_f) = \p$, then $\p$ corresponds to the unique tangent point between $\cA(\mbx_0;\gamma)$ and $\Upsilon$ due to the uniqueness of projection that follows from the convexity of TS and the circular shape of SR. Let us denote the initial position of $E$ in such case by $\x_E^{0*}$ (and the corresponding game state by $\mbx_0^*$), then from Proposition~\ref{prop:hamiltonian} and Lemma~\ref{lemma:apolsphere} we know that
    \begin{align} \label{eq:prop3proofeq1}
        \| \p - \x_E^{0*} \| = \gamma \| \p - \x_P^0 \|.
    \end{align}
    Squaring both sides of Eq.~\eqref{eq:prop3proofeq1} leads to an equation of the hypersphere that includes all candidate positions of $\x_E^{0*}$:
    \begin{align} \label{eq:prop3proofeq2}
        \left\|\x_E^{0*}\right\|^2 &- 2 \left\langle \x_E^{0*},\p \right\rangle - \gamma^2 \|\x_P\|^2 + \nonumber
        \\
        &\qquad\qquad (1-\gamma^2) \|\p\|^2 + 2\gamma^2 \left\langle \x_P^0,\p \right\rangle = 0.
    \end{align}
    Furthermore, the center of SR, $\bm{\alpha}(\mbx_0^*)$, belongs to the line that is orthogonal to $\partial \Omega$ at $\p$. Shifting this line towards $\x_P^0$ to align it with the possible positions of $E$, one can obtain
    \begin{align} \label{eq:prop3proofeq3}
        \x_E^{0*} = \nabla F(\p) \xi + \gamma^2 \x_P^0 + (1-\gamma^2) \p,
    \end{align}
    for some $\xi \in \bR_{\geq 0}$. Substituting Eq.~\eqref{eq:prop3proofeq3} into Eq.~\eqref{eq:prop3proofeq2} results in the following algebraic quadratic equation:
    \begin{align} \label{eq:prop3proofeq4}
        &\| \nabla F(\p) \|^2 \xi^2 + 2\gamma^2 \left \langle \x_P^0-\p,\nabla F(\p) \right \rangle \xi \nonumber
        \\
        &\qquad\qquad\qquad\qquad -\gamma^2(1-\gamma^2) \left\|\x_P^0-\p\right\|^2 = 0,
    \end{align}
    which always has the unique positive solution, $\xi^+ > 0$, since its discriminant is positive and the signs of its first and third coefficients are opposite, thereby establishing the bijectivity of the transformation. We finish the proof by substituting $\xi^+$ into Eq.~\eqref{eq:prop3proofeq3} to obtain Eq.~\eqref{eq:barriermap}.
\end{proof}


A non-trivial application of Eq.~\eqref{eq:barriermap} is shown in Figure~\ref{fig:barriers}, where the boundary of TS is given by a quartic surface for which Eq.~\eqref{eq:barrierfunction} admits no closed-form solution.




\section{Games of Degree} \label{sec:gamesofdegree}

In this section, we address the Capture and Attack Games of Degree by deriving optimal state feedback strategies in each player's winning set (whose open-loop representations are known from Proposition~\ref{prop:hamiltonian}). In particular, the optimality of these strategies is verified by the HJI equation \cite{bacsar1998dynamic} which for Mayer-type problems is given by
\begin{align} \label{eq:hjipde}
    -\partial V/\partial t = (\partial V/\partial \mbx) \mbf(\mbx,\uu_P^\star,\uu_E^\star),
\end{align}
where $V$ corresponds to Eq.~\eqref{eq:valfun1} for the Capture Game of Degree and Eq.~\eqref{eq:valfun2} for the Attack Game of Degree.

\subsection{Capture Game of Degree}

The Capture Game of Degree is first addressed for the case when $\mbx \in \cR_c$, that is, when capture is assured provided that $P$ plays optimally.

\begin{theorem} \label{theo:capgamedegree}
    Given the TPTDG defined in Section~\ref{sec:problem_formulation} and $\mbx \in \cR_c$, the value function $V_c : \cR_c \rightarrow \bR$ is $C^1$ (continuously differentiable) in $\cR_c$ and satisfies Eq.~\eqref{eq:hjipde}, where
    \begin{align}
        V_c(\mbx) &= \| \bm{\alpha}(\mbx)-\proj_\Omega (\bm{\alpha}(\mbx)) \| - \beta(\mbx). \label{eq:gamedegree_valuefunc_capture}
    \end{align}
    Furthermore, the optimal state feedback strategies for $P$ and $E$ in $\cR_c$, $\bm{\delta}_P^\star : \cR_c \rightarrow \cU$ and $\bm{\delta}_E^\star : \cR_c \rightarrow \cU$, are defined by
    \begin{equation}
        \bm{\delta}_P^\star(\mbx) =
        \dfrac{\x^\star - \x_P}{\| \x^\star - \x_P \|}, \quad \bm{\delta}_E^\star(\mbx) =
        \dfrac{\x^\star - \x_E}{\| \x^\star - \x_E \|}, \label{eq:capturegame_optinputs}
    \end{equation}
    where the optimal capture point $\x^\star$ is given by
    \begin{align}
        \x^\star &= \bm{\alpha}(\mbx) - \beta(\mbx) \frac{\bm{\alpha}(\mbx)-\proj_\Omega (\bm{\alpha}(\mbx))}{\| \bm{\alpha}(\mbx)-\proj_\Omega (\bm{\alpha}(\mbx)) \|}.
    \end{align}
\end{theorem}
\begin{proof}
    First, let us take the partial derivatives of $\bm{\alpha}(\mbx)$ and $\beta(\mbx)$ with respect to $\mbx$:
    \begin{align}
        \frac{\partial \bm{\alpha}(\mbx)}{\partial \mbx} &= \frac{1}{1-\gamma^2}
        \begin{bmatrix}
            - \gamma^2 I_{n\times n} & I_{n \times n}
        \end{bmatrix}, \label{eq:capturegame_Cpartial}
        \\
        \frac{\partial \beta(\mbx)}{\partial \mbx} &= \frac{\gamma}{1-\gamma^2}
        \begin{bmatrix}
            -\frac{(\x_E - \x_P)^\top}{\| \x_E - \x_P \|} & \frac{(\x_E - \x_P)^\top}{\| \x_E - \x_P \|}
        \end{bmatrix} \label{eq:capturegame_Rpartial}
    \end{align}
    In Theorem~\ref{theo:barriersurface}, the value function has been derived as
    \begin{align}
        V_c(\mbx) = \| \bm{\alpha}(\mbx)-\proj_\Omega (\bm{\alpha}(\mbx)) \| - \beta(\mbx).
    \end{align}
    Note that the function $\mbx \mapsto \|\bm{\alpha}(\mbx)-\z\|$ is continuous and convex on $\cR_c$ for any given $\z \in \Omega$, and the projection $\proj_\Omega (\bm{\alpha}(\mbx))$ is unique since TS is closed and convex. Let $\bm{g}(\bm{\alpha}(\mbx)) := \bm{\alpha}(\mbx)-\proj_\Omega (\bm{\alpha}(\mbx))$ for brevity, then by virtue of Danskin's theorem~\cite{bertsekas2003convex}, it follows that
    \begin{align} \label{eq:grad_Vc}
        \frac{\partial V_c(\mbx)}{\partial \mbx} &= \frac{\partial \| \bm{g}(\bm{\alpha}(\mbx)) \|}{\partial \mbx} - \frac{\partial \beta(\mbx)}{\partial \mbx} \nonumber
        \\
        &= \left[ \frac{\partial \| \bm{g}(\bm{\alpha}(\mbx)) \|}{\partial \bm{\alpha}(\mbx)} \right]^\top ~ \frac{\partial \bm{\alpha}(\mbx)}{\partial \mbx} - \frac{\partial \beta(\mbx)}{\partial \mbx} \nonumber
        \\
        &= \left[ \frac{\bm{g}(\bm{\alpha}(\mbx))}{\| \bm{g}(\bm{\alpha}(\mbx)) \|} \right]^\top \frac{\partial \bm{\alpha}(\mbx)}{\partial \mbx} - \frac{\partial \beta(\mbx)}{\partial \mbx} \nonumber
        \\
        &= \frac{1}{1-\gamma^2}
        \begin{bmatrix}
            -\gamma^2 \frac{\bm{g}(\bm{\alpha}(\mbx))}{\|\bm{g}(\bm{\alpha}(\mbx))\|} + \gamma \frac{\x_E-\x_P}{\| \x_E - \x_P \|}
            \\
            \frac{\bm{g}(\bm{\alpha}(\mbx))}{\|\bm{g}(\bm{\alpha}(\mbx))\|} - \gamma \frac{\x_E-\x_P}{\| \x_E - \x_P \|}
        \end{bmatrix}^\top.
    \end{align}
    Since its partial derivatives exist and are continuous over $\cR_c$, where $(\x_P \neq \x_E) \wedge (\bm{\alpha}(\mbx) \neq \proj_\Omega(\bm{\alpha}(\mbx))$, $V_c$ is $C^1$ in the same set. Next, using the facts that $\x_P$, $\x_E$, and $\bm{\alpha}(\mbx)$ are colinear, i.e., $(\bm{\alpha}(\mbx)-\x_P)/\|\bm{\alpha}(\mbx)-\x_P\| = (\x_E-\x_P)/\|\x_E-\x_P\|$, and that $\|\bm{\alpha}(\mbx)-\x_E\|=\gamma^2\|\bm{\alpha}(\mbx)-\x_P\|=\gamma \beta(\mbx)$, \eqref{eq:capturegame_optinputs} can be written as
    \begin{equation} \label{eq:rewritten_inputs}
    \begin{aligned}
        \bm{\delta}_P^\star(\mbx) &= \beta(\mbx) \frac{\frac{\x_E-\x_P}{\|\x_E-\x_P\|}-\gamma \frac{\bm{g}(\bm{\alpha}(\mbx))}{\|\bm{g}(\bm{\alpha}(\mbx))\|}}{\gamma\| \x^\star - \x_P \|},
        \\
        \bm{\delta}_E^\star(\mbx) &= \beta(\mbx) \frac{\gamma\frac{\x_E-\x_P}{\|\x_E-\x_P\|}-\frac{\bm{g}(\bm{\alpha}(\mbx))}{\|\bm{g}(\bm{\alpha}(\mbx))\|}}{\| \x^\star - \x_P \|}.
    \end{aligned}
    \end{equation}
    Finally, we have that $\partial V_c/\partial t = 0$ since $V_c$ is time-invariant. Substituting \eqref{eq:grad_Vc} and \eqref{eq:rewritten_inputs} into the RHS of Eq.~\eqref{eq:hjipde} yields
    \begin{align}
        &(\partial V_c / \partial \mbx)  \mbf(\mbx,\bm{\delta}_P^\star,\bm{\delta}_E^\star) = (\partial V_c / \partial \mbx) v_P \left[ \bm{\delta}_P^{\star \top} ~ \gamma \bm{\delta}_E^{\star \top} \right]^\top \nonumber
        \\
        &= \frac{v_P \beta(\mbx)}{(1-\gamma^2)\|\x^\star - \x_P\|} \cdot \nonumber
        \\
        &\quad
        \begin{bmatrix}
            - \frac{\gamma^2\bm{g}(\bm{\alpha}(\mbx))}{\|\bm{g}(\bm{\alpha}(\mbx))\|} + \frac{\gamma(\x_E-\x_P)}{\| \x_E - \x_P \|}
            \\
            \frac{\bm{g}(\bm{\alpha}(\mbx))}{\|\bm{g}(\bm{\alpha}(\mbx))\|} - \frac{\gamma(\x_E-\x_P)}{\| \x_E - \x_P \|}
        \end{bmatrix}^\top
        \begin{bmatrix}
            \frac{\x_E-\x_P}{\gamma\|\x_E-\x_P\|}-\frac{\bm{g}(\bm{\alpha}(\mbx))}{\|\bm{g}(\bm{\alpha}(\mbx))\|}
            \\
            \frac{\gamma(\x_E-\x_P)}{\|\x_E-\x_P\|}-\frac{\bm{g}(\bm{\alpha}(\mbx))}{\|\bm{g}(\bm{\alpha}(\mbx))\|}
        \end{bmatrix} \nonumber
        \\
        &= 0.
    \end{align}
    This completes the proof.
\end{proof}

\begin{remark}
    Since $V_{c}(\mbx)$ is $C^1$ on $\cR_c$, there is no dispersal surface in $\cR_c$ and thus, the optimal state feedback strategies $\bm{\delta}_P^\star(\cdot)$ and $\bm{\delta}_E^\star(\cdot)$ correspond to the unique saddle point of $J_c$. These strategies have closed-form expressions when TS is simple; otherwise they can be computed by solving the optimization problem (projection) numerically.
    \label{remark:capgame_semiperm}
\end{remark}

\subsection{Attack Game of Degree}
Next, the Attack Game of Degree is addressed for the case where $\mbx \in \cR_a$, i.e., attack is ensured if $E$ plays optimally.

\begin{theorem} \label{theo:attgamedegree}
    Given the TPTDG defined in Section~\ref{sec:problem_formulation} and $\mbx \in \cR_a$, the value function $V_a : \cR_a \rightarrow \bR$ is $C^1$ on $\cR_a$ and satisfies Eq.~\eqref{eq:hjipde}, where
    \begin{equation} \label{eq:gamedegree_valuefun_attack}
        V_a(\mbx) = -\| \x^\dagger - \x_P \| + \| \x^\dagger - \x_E \| / \gamma.
    \end{equation}
    The optimal state feedback strategies for $P$ and $E$ in $\cR_a$, $\bm{\theta}_P^\star : \cR_a \rightarrow \cU$ and $\bm{\theta}_E^\star : \cR_a \rightarrow \cU$, are defined by
    \begin{equation}
        \bm{\theta}_P^\star(\mbx) =
        \dfrac{\x^\dagger - \x_P}{\| \x^\dagger - \x_P \|}, \quad \bm{\theta}_E^\star(\mbx) =
        \dfrac{\x^\dagger - \x_E}{\| \x^\dagger - \x_E \|},
        \label{eq:attackgame_optinputs}
    \end{equation}
    where the optimal attack point $\x^\dagger$ is given by
    \begin{equation}\label{eq:attackgame_optattpoint}
       \x^\dagger = \argmin_{\z \in \mathrm{cl}(\cA(\mbx;\gamma)) \cap \Omega} -\| \z - \x_P \| + \| \z - \x_E \|/\gamma.
   \end{equation}
\end{theorem}
\begin{proof}
    In view of Proposition~\ref{prop:hamiltonian}, Eq.~\eqref{eq:valfun2} becomes
    \begin{align} \label{eq:Vader}
        &V_a(\mbx) = \min\nolimits_{\bm{\theta}_P^\star(\cdot)}\max\nolimits_{\bm{\theta}_E^\star(\cdot)} \| \x_E(t_f) - \x_P(t_f) \| \nonumber
        \\
        &= \min\nolimits_{\z \in \mathrm{cl}(\cA(\mbx;\gamma)) \cap \Omega}~ -\| \z - \x_P \| + \| \z - \x_E \| / \gamma.
    \end{align}
    Since $\mathrm{cl}(\cA(\mbx;\gamma)) \cap \Omega$ is the non-empty intersection of two convex and closed sets, there exists a unique $\x^\dagger$ that minimizes the function $\varphi(\mbx, \z):=-\| \z - \x_P \| + \| \z - \x_E \| / \gamma$ in $\mathrm{cl}(\cA(\mbx;\gamma)) \cap \Omega$ for any given $\mbx \in \cR_a$. Additionally, the function $\mbx \mapsto \varphi(\mbx,\z)$ is continuous and convex on $\cR_a$ for any given $\z \in \mathrm{cl}(\cA(\mbx;\gamma)) \cap \Omega$. Hence,
    \begin{align}
        V_a(\mbx) = -\| \x^\dagger - \x_P \| + \| \x^\dagger - \x_E \| / \gamma.
    \end{align}
    The derivative of $V_a$ is given by
    \begin{align} \label{eq:grad_Va}
        \frac{\partial V_a(\mbx)}{\partial \mbx} &=
        \begin{bmatrix}
            -(\x^\dagger - \x_P)/\| \x^\dagger - \x_P \|
            \\
            \x^\dagger - \x_E/(\gamma\| \x^\dagger - \x_E \|)
        \end{bmatrix}^\top,
    \end{align}
    where $(\x^\dagger \neq \x_P) \wedge (\x^\dagger \neq \x_E)$ in $\mbx \in \cR_a$. Since the partial derivatives of $V_a$ exist and are continuous over $\cR_a$, $V_a$ is $C^1$ on the same set. Finally, we again have $\partial V_a/\partial t = 0$, and substituting Eq.~\eqref{eq:grad_Va} and Eq.~\eqref{eq:attackgame_optinputs} into Eq.~\eqref{eq:hjipde} we obtain
    \begin{align}
        &(\partial V_a / \partial \mbx)  \mbf(\mbx,\bm{\theta}_P^\star,\bm{\theta}_E^\star) = (\partial V_a / \partial \mbx) v_P \left[\bm{\theta}_P^{\star \top} ~ \gamma \bm{\theta}_E^{\star \top} \right]^\top \nonumber
        \\
        &\quad = v_P
        \begin{bmatrix}
            -\frac{\x^\dagger - \x_P}{\| \x^\dagger - \x_P \|}
            &
            \frac{\x^\dagger - \x_E}{\gamma \| \x^\dagger - \x_E \|}
        \end{bmatrix}
        \begin{bmatrix}
            \frac{\x^\dagger - \x_P}{\| \x^\dagger - \x_P \|}
            \\
            \frac{\gamma(\x^\dagger - \x_E)}{\| \x^\dagger - \x_E \|}
        \end{bmatrix} = 0.
    \end{align}
    This completes the proof.
\end{proof}

\begin{remark}
    Since $V_a$ is $C^1$ on $\cR_a$, there exists no dispersal surface in $\cR_a$, and thus the optimal state feedback strategies $\bm{\theta}_P^\star(\cdot)$ and $\bm{\theta}_E^\star(\cdot)$ correspond to the unique saddle point of $J_a$. The closed-form expressions of these strategies can be found if TS has a simple shape, otherwise they can be computed numerically.
    \label{remark:attgame_semiperm}
\end{remark}

\section{Numerical Simulations} \label{sec:simulations}

\begin{figure}
    \centering
    \vspace{3mm}
    \includegraphics[scale=0.28]{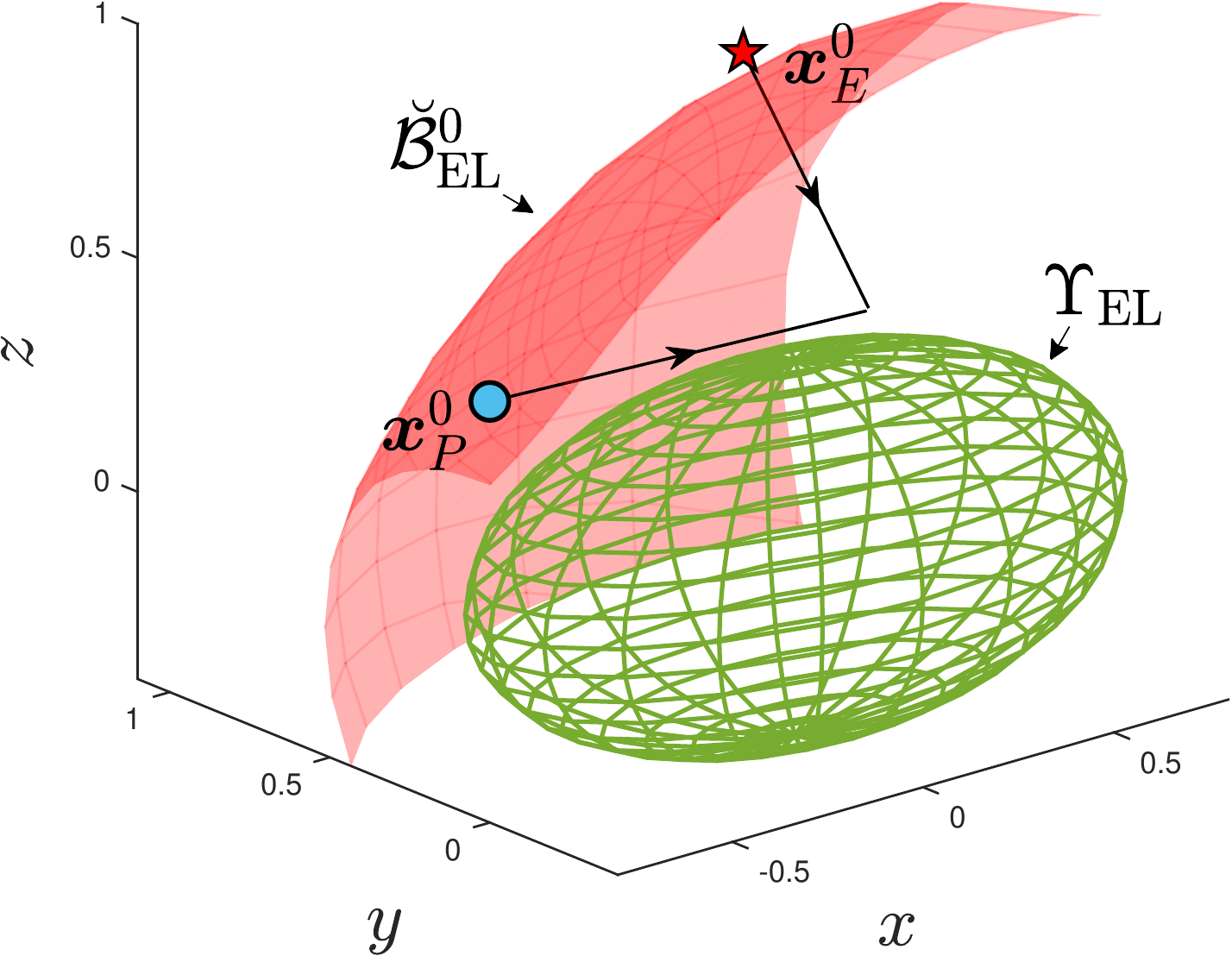}
    \includegraphics[scale=0.28]{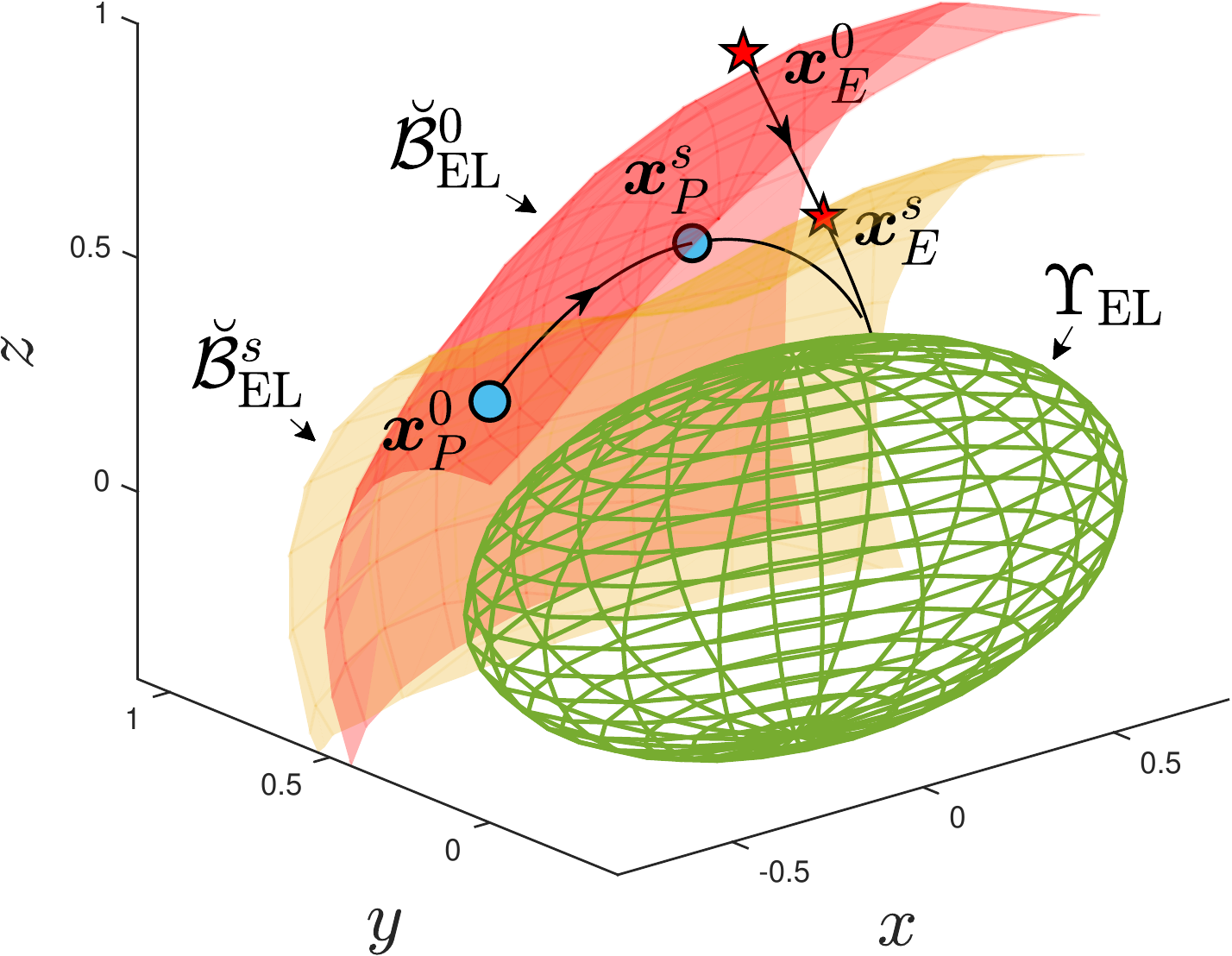}
    \includegraphics[scale=0.22]{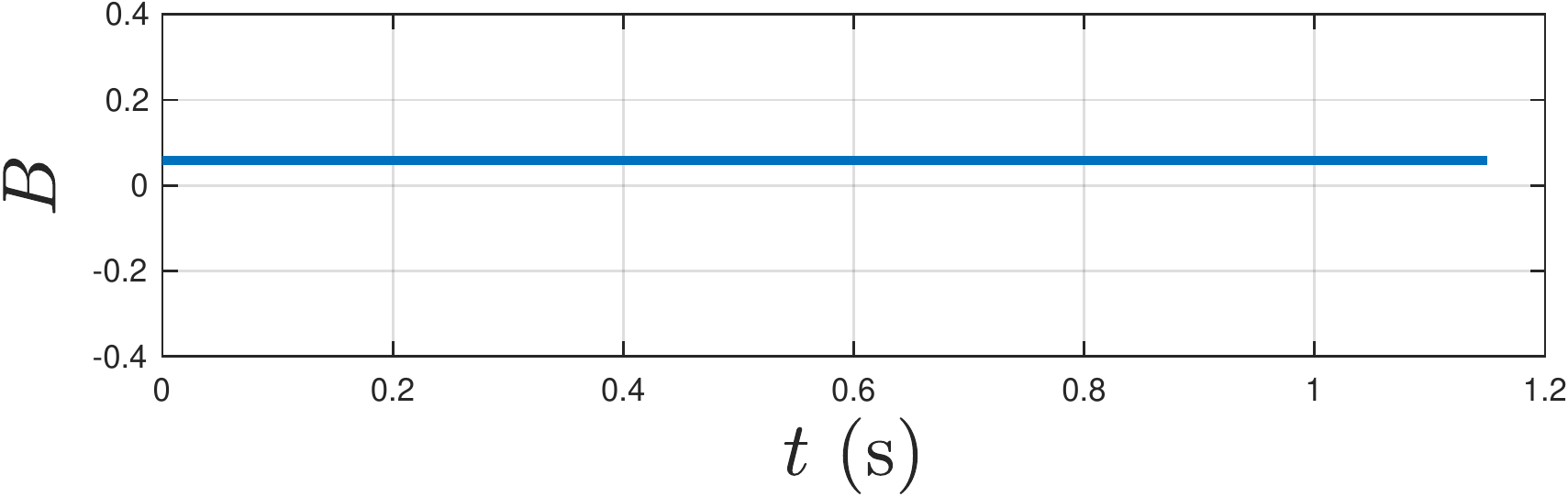}
    \includegraphics[scale=0.22]{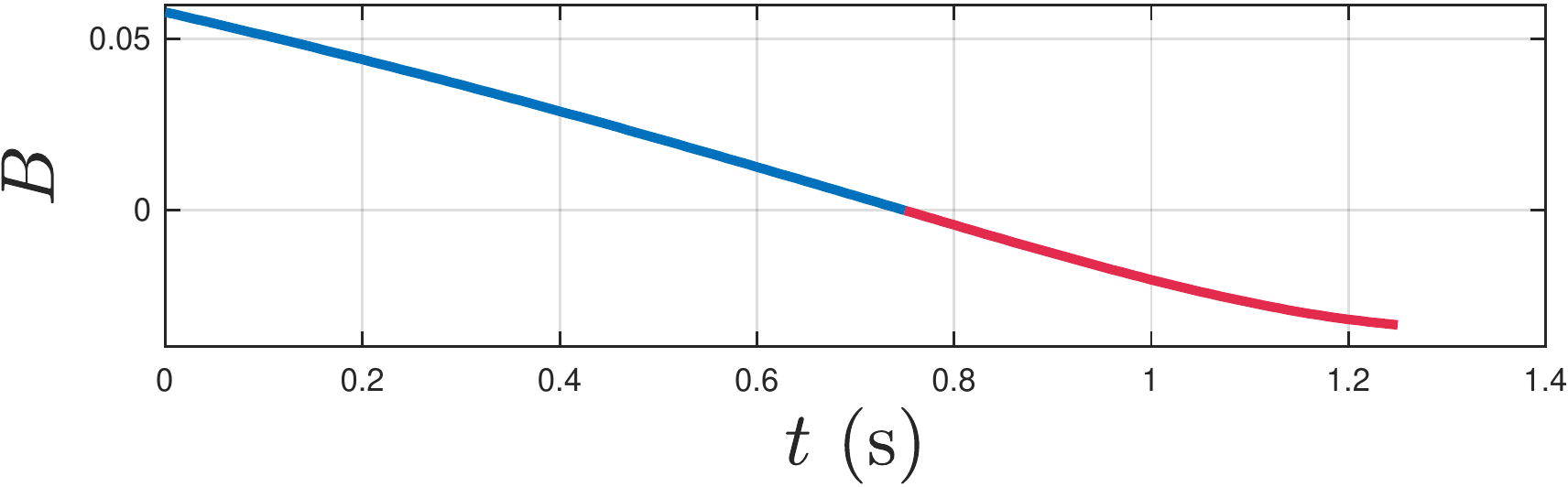}
    \caption{(Left) trajectories under optimal play and value of the barrier function, (right) trajectories of players under $P$'s non-optimal play and value of the barrier function.} 
    \label{fig:simCG}
\end{figure}

In this section, we present numerical simulations for four different scenarios of the TPTDG. In Figures~\ref{fig:simCG} and \ref{fig:simAG}, TS (green ellipsoid) corresponds to the (compact) set $\Upsilon_\mathrm{EL} := \{\z \in \bR^3 : x^2/0.8^2 + y^2/0.4^2 + z^2/0.4^2 \leq 1\}$. The initial position of $P$ (blue circle) is $\x_P^0 = (-0.8,0,0.5)$, whereas the initial positions of $E$ (red star) are $\x_E^0 = (0.2,0.4,0.9)$ in Figure~\ref{fig:simCG} and $\x_E^0 = (0.2,0.2,0.7)$ in Figure~\ref{fig:simAG}, respectively. The speed ratio is always $\gamma = 0.5$. The PBS $\breve \cB_\mathrm{EL}^0$ (red surface) and $\breve \cB_\mathrm{EL}^s$ (yellow surface) are the projected images of the 6-th dimensional barrier surface $\cB$ onto the game space $\bR^3$ with respect to $\x_P^0$ and $\x_P^s := \x_P(t_s)$, respectively, where $t_s$ is the switching time at which the game state $\mbx$ crosses $\cB$. PBS is constructed via Eq.~\eqref{eq:barriermap} as the ellipsoidal TS admits no closed-form projection operator. All the control strategies are computed numerically at every time instant.

In Figure~\ref{fig:simCG}, the players begin by playing the capture game since $\x_E^0$ lies outside $\breve \cB_\mathrm{EL}^0$ (i.e., $\mbx_0 \in \cR_c$). On the left, both players play optimally (i.e., employ strategies given in \eqref{eq:capturegame_optinputs}), whereas on the right $P$ employs the (suboptimal) pure-pursuit strategy~\cite{bakolas2012relay}. On the left, the trajectories of both players are straight lines, $E$ is eventually captured by $P$, and the value of the barrier function $B$ stays above zero during the game. Conversely, on the right, $E$ crosses the barrier at $t_s$ (when $B=0$), switches her strategy to Eq.~\eqref{eq:attackgame_optinputs}, and successfully attacks TS.

In Figure~\ref{fig:simAG}, $E$ is initially deployed in $\breve \cB_\mathrm{EL}^0$ (i.e., $\mbx_0 \in \cR_a$), so the players begin by playing the attack game. On the left, both players play optimally (i.e., strategies given in \eqref{eq:attackgame_optinputs}), whereas on the right $E$ directly heads towards an arbitrary point on TS (suboptimal strategy). On the left, $P$, knowing that capture is impossible, strives to minimize his distance from $E$ who eventually enters TS, resulting in straight-line trajectories. On the right, $E$ crosses $\cB$ at $t_s$ (when $B= 0$) after which $P$ switches his strategy to Eq.~\eqref{eq:capturegame_optinputs}. Finally, $P$ captures $E$ outside TS.

\begin{figure}
    \centering
    \vspace{3mm}
    \includegraphics[scale=0.28]{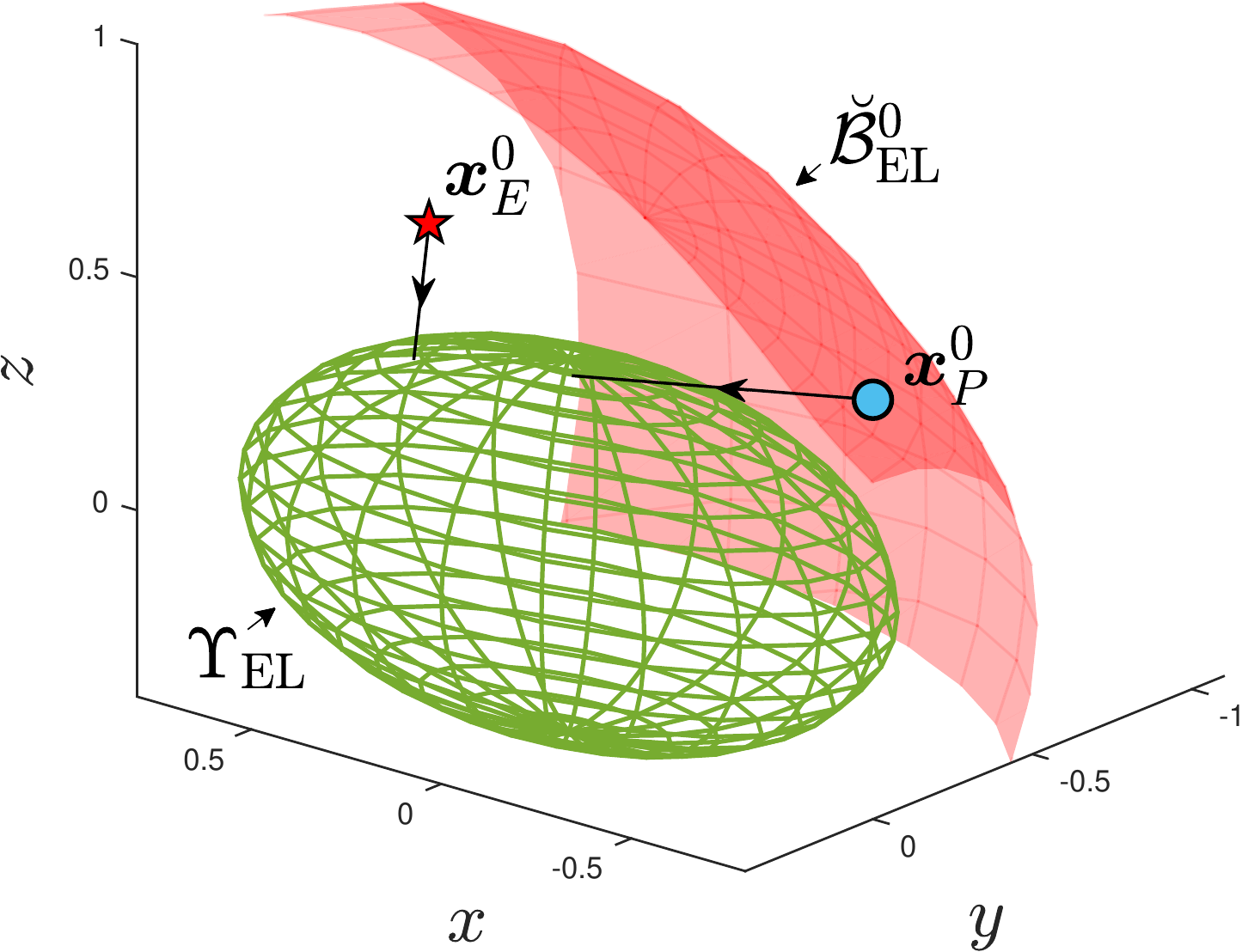}
    \includegraphics[scale=0.28]{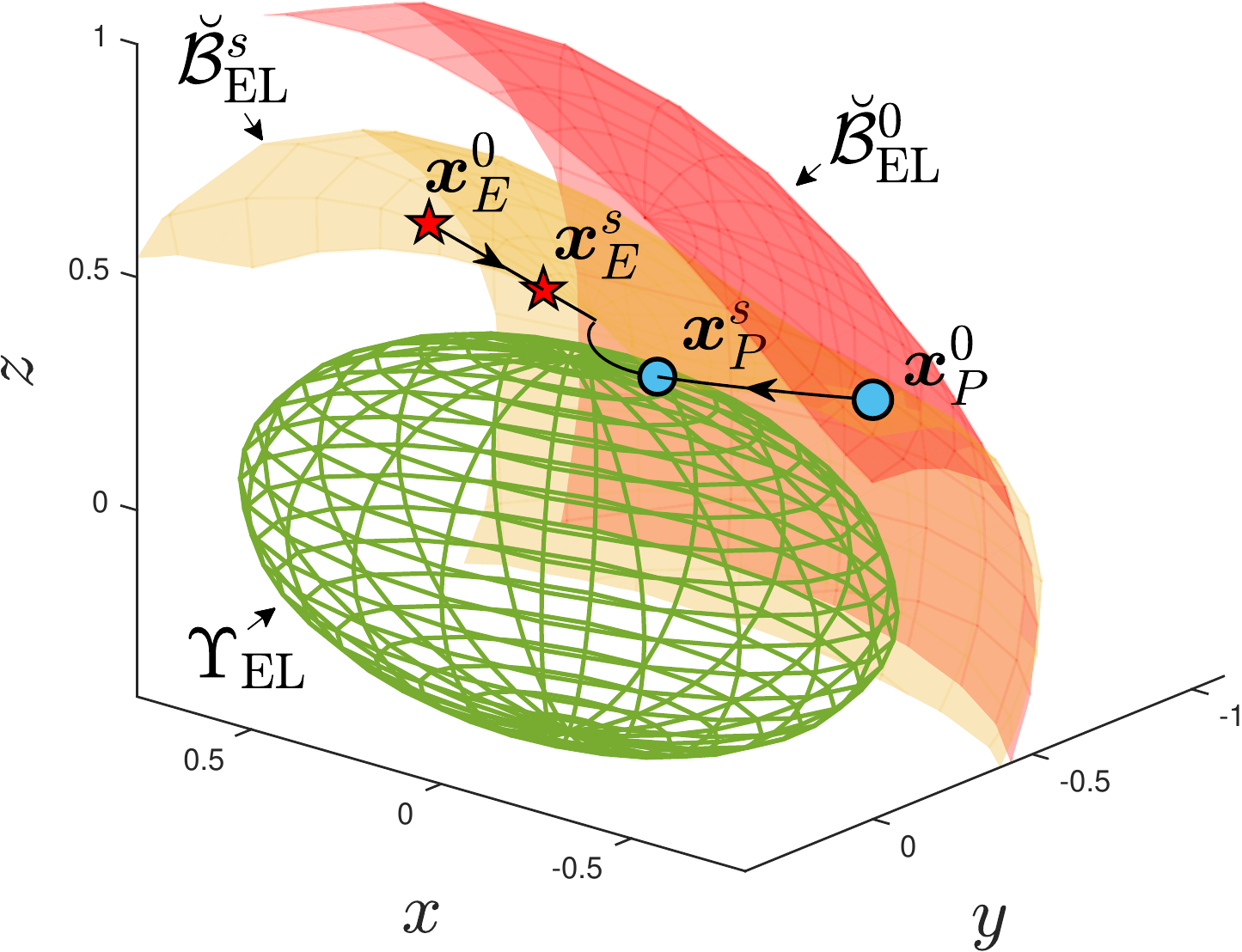}
    \includegraphics[scale=0.22]{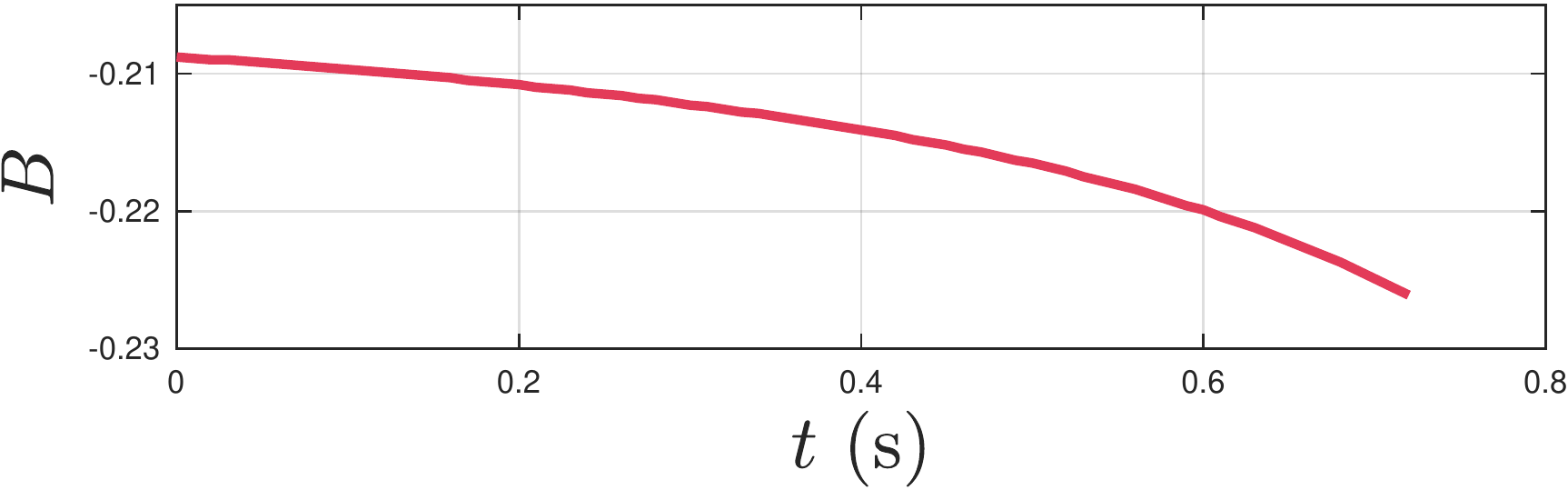}
    \includegraphics[scale=0.22]{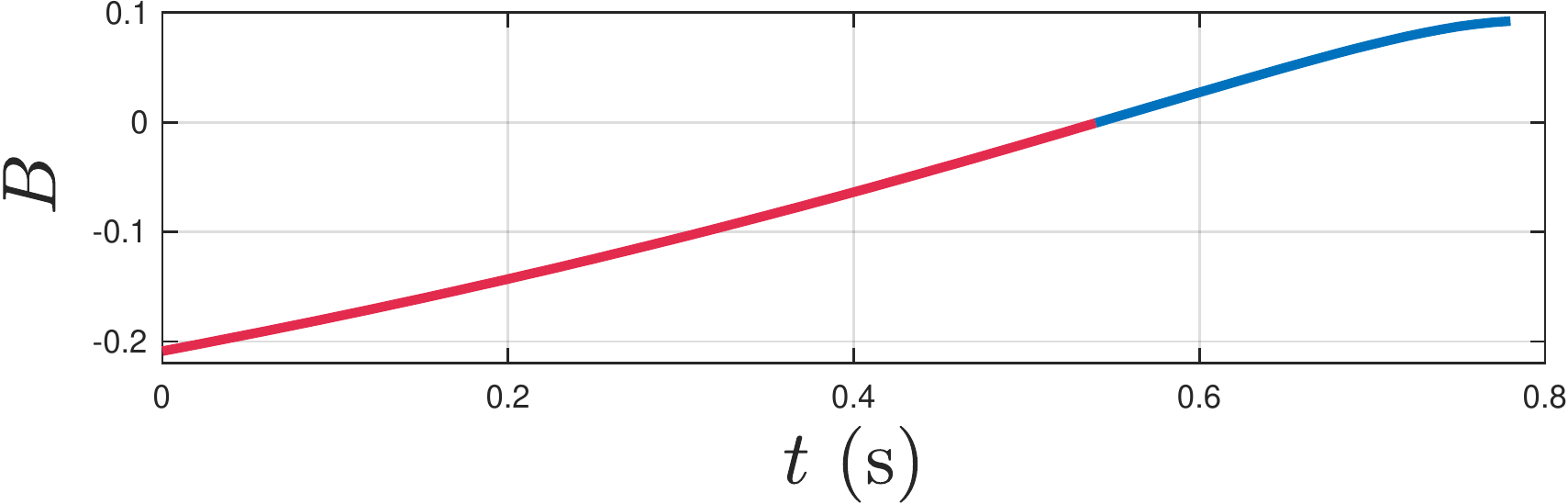}
    \caption{(Left) trajectories under optimal play and value of the barrier function, (right) trajectories under $E$'s non-optimal play and value of the barrier function.}
    \label{fig:simAG}
\end{figure}

\section{Conclusion} \label{sec:conclusion}
In this paper, we have analyzed the problem of guarding a closed and convex target set from a single attacker in the $n$-th dimensional Euclidean space based on differential game theory. Our solution can be applicable to not only cases studied in previous related work but also more advanced cases where no analytical form of the barrier or the optimal strategies may exist. Numerical simulations results have been presented to verify the efficacy of our solution in such cases. In our future work, we will extend the proposed solution approach to multi-player target defense games.

\bibliographystyle{ieeetr}
\bibliography{pegref}

\end{document}